\documentclass[draft]{amsart}

\newtheorem{thm}{Theorem}[section]
\newtheorem{cor}[thm]{Corollary}
\newtheorem{lem}[thm]{Lemma}
\newtheorem{prop}[thm]{Proposition}
\theoremstyle{definition}

\theoremstyle{remark}
\newtheorem{rem}[thm]{Remark}
\theoremstyle{definition}
\newtheorem{ex}[thm]{Example}

\newcommand{\CC}{\mathbb{C}}

\newcommand{\PP}{\mathbb{P}}

\newcommand{\reviewtimetoday}[2]{
}
\reviewtimetoday{\today}{Draft Version}

\usepackage{xypic}

\begin{document}

\title{On rational normal curves in projective space}

\author{E. Carlini, M. V. Catalisano}

\address{Enrico Carlini, Dipartimento di Matematica del Politecnico di Torino, Corso Duca degli Abruzzi 24, 10129 Turin, Italy.}
\email{enrico.carlini@polito.it}

\address{Maria Virginia Catalisano, DIPTEM - Dipartimento di Ingegneria della Produzione, Termoenergetica e Modelli
Matematici, Piazzale Kennedy, pad. D 16129 Genoa, Italy.}
\email{catalisano@dipem.unige.it}


\date{}


\begin{abstract}
In this paper we consider a generalization of a well known result
by Veronese about rational normal curves. More precisely, given a
collection of linear spaces in $\PP^n$ we study the existence of
rational normal curves intersecting each component of the
configuration maximally. We introduce different methods to show
existence and non-existence of such curves. We also show how to
apply these techniques to the study of defectivity of
Segre-Veronese varieties.
\end{abstract}

\maketitle
\section{Introduction}
The starting point of our study is the following well know result
\begin{thm}[Veronese]\label{Veronese} Given $n+3$ generic points in $\PP^n$ there
exists a unique rational normal curve passing through them.
\end{thm}
Although often attributed to Castelnuovo, this theorem was proved
by Veronese \cite{Veronese}. Alternative proofs can be produced
using a constructive approach as shown in \cite[page 10]{Harris}
or via Cremona transformations as recalled in \cite[Theorem
4.1]{CaCa}. In fact, a new proof can be given using the methods of
this paper, see Remark \ref{VERONESEremark}.

Veronese's result can be generalized in different ways. For
example, one can consider more general varieties passing through
points, e.g. $d$-uple embeddings of the projective plane.
Kapranov, for example, studied a possible generalization of
Theorem \ref{Veronese} to Veronese surfaces in \cite{Kapranov},
see also \cite{GraberRanestad} for another proof of Kapranov's
result. A strictly related problem is treated in \cite{EiHuPo}
where the maximal number of intersection points of two Veronese
surfaces is determined. In a different direction, one can still
consider rational normal curves, but now satisfying more general
conditions. In the early twentieth century, for example, some
interest focused on rational normal curves passing through points
and intersecting codimension two linear spaces of $\PP^n$ in $n-1$
points, i.e. in the maximal number of points allowed, see
\cite[page 221]{Room},\cite{Wakeford} and \cite{Babbage}. In
\cite{CaCa} we considered this precise situation and we were able
to recover all the known results and to prove new ones. Our
analysis yielded the following:
\begin{thm}[Theorem 4.7 in \cite{CaCa}]\label{NOSTROcod2}
Let $n,p$ and $l$ be non-negative integers such that
$$n\geq
3, \ \ p\geq 1 \mbox{ and} \ \ p+l=n+3 .$$ Choose  $p$ generic
points in $\PP^n$ and $l$ generic codimension two linear spaces.
Then, for the values
\[
(p,l)=(n+3,0),(n+2,1),(3,n),(2,n+1),(1,n+2)
\]
does there exist a unique rational normal curve passing through
the points and $(n-1)$-secant to the linear spaces. In the other
cases, that is for
$p\geq 4$ and $l\geq 2$, no such curve exists.
\end{thm}
As an application of this result, we studied higher secant
deficiencies for some Segre-Veronese varieties, see \cite{CaCa}.
For more on Segre-Veronese varieties see \cite{CGG3,CGG4} and for
more on defective varieties \cite{Zak, Ge, Ci01}. In \cite{CaCa},
we also noticed that rational normal curves satisfying more
general condition could be useful, see Remark
\ref{SEGREVERONESErem}. More precisely, it is interesting and
useful to consider the following question
\begin{quote}{\it
given a finite collection of linear spaces in generic position in
$\PP^n$, is there a rational normal curve intersecting each
$i$-dimensional component of the collection in $i+1$ points?}
\end{quote}

This paper is devoted to the study of this last question. Giving a
complete answer appears to very ambitious and quite difficult and
so our result are only partial. Nevertheless, they do give some
understanding of the situation. For example, Corollary
\ref{finitelyCOR} shows that for any $n$ there are only finitely
many cases to consider. While Theorem \ref{SEGREteoiff} and
Proposition \ref{bezoutPROP} produce an answer in some of the
remaining cases. To understand the intrinsic difficulty of the
question, we notice that even when all the spaces have dimension
equal to one, there are open cases, see Subsection
\ref{HOMOGENEOUSsec}.

The relevant information to describe a generic collection of
linear spaces $\Lambda=\cup\Lambda_i$ can be encoded in a vector
of non-negative integers $L=(l_0,\ldots,l_{n-2})$ where $l_i$ is
the number of $i$-dimensional spaces in $\Lambda$. Now, given $L$,
should we expect a positive or a negative answer to our question?
We can develop some heuristic machinery to help our intuition.
Using the matrix representation of a rational normal curve, one
can easily see that the rational normal curves passing through a
given point are a codimension $n-1$ family in their natural
parameter space. The rational normal curves 2-secant to a given
line are a family of codimension $2(n-2)$. Similarly,  the family
of rational normal curves 3-secant to a given plane has
codimension $3(n-3)$. In conclusion, as the parameter space for
rational normal curves in $\PP^n$ has dimension $(n+3)(n-1)$,
given $\Lambda$ with relevant vector $(l_0,\ldots,l_{n-2})$, we
expect there to be a family, of rational normal curves, which has
dimension
\[(n+3)(n-1)-\sum_{i=0}^{n-2}l_i(i+1)(n-1-i),\]
to be $(i+1)$-secant to each $i$-dimensional component of
$\Lambda$. This kind of argument will be made rigorous and used to
derive results in Subsection \ref{paramterSUBSEC}. Since the
question is not clearly linear in nature, the actual dimension of
the family can very well be bigger or smaller than the expected
one.

In Section \ref{SECdef}, we introduce some notation for
configurations of linear spaces and their relevant vectors.

In Section \ref{SECfeasible}, we introduce methods to show
existence of rational normal curves.

In Section \ref{SECnonfeasible}, we describe methods to prove
non-feasibility.

In Section \ref{SECapp}, we apply the methods developed in the
paper to  study  some special situations and to produce results
about Segre-Veronese varieties and the behavior of their higher
secant varieties. More precisely, we study homogeneous
configurations, where all the components have the same dimension.
In the simplified situation when all linear spaces have dimension
equal to one, we have a complete answer except for five lines in
$\PP^5$ and six lines in $\PP^7$, see Proposition \ref{lineprop}.
We also study the configurations consisting of exactly one space
in each dimension, giving an answer for all cases, except for
$\PP^6$ and $\PP^7$, see Proposition \ref{onespaceineachPROP}. We
present a family of non-feasible configurations, not containing
points, which stems from the analysis of the lines case. Finally,
we show how to apply our results to find new defectivity results
for the Segre-Veronese varieties $\PP^1 \times \PP^m \times \PP^m$
embedded with multi-degree $(2,1,1)$. For more on Segre-Veronese
varieties see \cite{CGG3,CGG4}.

\begin{rem} Our question seems to be suitable for an Intersection
Theory approach and we posed the question to many experts in that
field. However, as far as we know, there is no
``ready-to-use-machinery" which can easily produce explicit
results.
\end{rem}

The authors wish to thank Gianfranco Casnati for stimulating
further investigations on the map used in Subsection
\ref{SEGREsection}. We also want to acknowledge Tony Geramita's
suggestion to look further into the lines case. His suggestion led
to Subsection \ref{HOMOGENEOUSsec}.

\section{Definition}\label{SECdef}

We work over the field of complex numbers $\mathbb{C}$. A rational
normal curve in $\PP^n$ will be denoted as an {\em rnc} in
$\PP^n$.

A vector of non-negative integers $L=(l_0,\ldots,l_{n-2})$ is
called a {\it weight vector in $\PP^n$}. A scheme
$\Lambda\subset\PP^n$ is said to be a {\it configuration of linear
spaces} of {\it weight} $L$ if $\Lambda$ is the union of $l_i$
linear spaces of dimension $i$, for $i=0,\ldots,n-2$; for we will
denote this configuration as $\Lambda_L$. We say that $\Lambda_L$
is generic if the linear components of $\Lambda$ are chosen in a
suitable open non-empty subset of a product of Grassmannians.

\begin{rem} There is a great deal of literature dealing with
{\em subspace arrangements} and these are nothing more than the
affine interpretation of what we call a configuration of linear
spaces, e.g. see \cite{Deb},\cite{BjPee} and \cite{Sid}. In
particular, in \cite{Derksen} the Hilbert polynomial of a generic
subspace arrangements is determined and this is a first step in
the direction of the study of Hilbert function. This is relevant
to us as the knowledge of the Hilbert function of a configuration
of linear spaces is needed for the application of Proposition
\ref{bezoutPROP}.
\end{rem}

We say that a configuration of linear spaces $\Lambda = \Lambda_1
\cup \dots \cup \Lambda_t \subset\PP^n$ is {\it feasible} if there
exists a rnc $\mathcal{C}\subset\PP^n$ maximally intersecting it,
i.e. if deg($ \mathcal{C}\cap \Lambda_i$) = $\dim \Lambda_i +1$,
for any  component $\Lambda_i$ of $\Lambda$.

We say that the vector $L$ is {\it feasible} if and only if a generic
configuration of linear spaces of weight $L$ is feasible.

\begin{rem} With this notation we can rephrase Theorem \ref{Veronese}as
follows:
\begin{quote}{\em
The weight vector $(p,0,\ldots,0)$ in $\PP^n$ is feasible if and
only if $p\leq n+3$. }\end{quote}

Similarly, Theorem \ref{NOSTROcod2} reads as follows:
\begin{quote}{\em
The weight vector $(p,0,\ldots,l),p\geq 1,$ in $\PP^n$ is feasible
if and only if $p\leq 4$ or $l\leq 2$. }\end{quote}
\end{rem}

\section{Methods for feasibility}\label{SECfeasible}

\subsection{Veronese's Theorem}
As a straightforward application of Theorem \ref{Veronese} we get
a useful numerical condition for feasibility.

\begin{prop}\label{numericalEXIST}
Let $L=(l_0,\ldots,l_{n-2})$ be a weight vector in $\PP^n$. If
\[\sum_{i=0}^{n-2} (i+1)l_i\leq n+3,\]
then $L$ is feasible.
\end{prop}
\begin{proof}
Given a generic configuration of linear spaces $\Lambda_L$, choose
$i+1$ generic points on each $i$-dimensional component of
$\Lambda_L$ and  apply Theorem \ref{Veronese}.
\end{proof}

\begin{rem} Notice that the numerical condition of the previous
Proposition is sufficient, but not necessary. Consider, for example, the
weight vector $L=(5,1)$ in $\PP^3.$
By Theorem \ref{NOSTROcod2} we know that $L$ is feasible, see
\cite{CaCa} for more details. However,
\[\sum_{i=0}^{n-2} (i+1)l_i=7> n+3=6\]
and feasibility can not be deduced by Proposition \ref{numericalEXIST}.
\end{rem}

\subsection{Segre varieties}\label{SEGREsection}
In this subsection, we use Segre varieties as a tool to verify
feasibility. We begin with a study of Segre varieties and then we
show how they are involved in the study of rncs. In particular, we
start from the description of an explicit parametrization of these
varieties which can be found in  \cite[pg.84]{CGG3}. Then, we
study the existence of rational curves with special properties on
the Segre varieties. Coupling these results we can prove
feasibility results.

We now fix notation for this section: for positive integers
$n_1\leq n_2\ldots\leq n_r$, we let $n=n_1+ n_2\ldots+ n_r$.
Consider the $n$ dimensional abstract Segre variety
\[\mathbf{S}=\PP^{n_1}\times\ldots\times\PP^{n_r}\]
and its Segre embedding
\[\mathbf{S}\hookrightarrow\mathcal{S}\subset\PP^N\]
where $N=(n_1+1)\cdot\ldots\cdot(n_r+1)-1$.

Given hyperplanes
\[H_1\subset\PP^{n_1},\ldots,H_r\subset\PP^{n_r}\]
we denote by $\mathbf{T}$ the Segre variety
\[\mathbf{T}=H_1\times\ldots\times H_r\subset\mathbf{S},\]
and by  $\mathcal{I}_\mathbf{T}$  its ideal sheaf.
Consider $(\mathcal{I}_\mathbf{T})^{r-1}$ in multidegree
$(1,\ldots,1)$, let
\[
\mathcal{L}=
H^0(\mathbf{S},
(\mathcal{I}_\mathbf{T})^{r-1}(1,\ldots,1)) ,
\]
and let $\phi_\mathbf{T}$ be the standard map:
\[\phi_\mathbf{T}:\mathbf{S}\rightarrow\PP(\mathcal{L})^{*} .\]
Recall that the
Chow ring of the Segre variety is
\[CH(\mathbf{S})={\mathbb{Z}[h_1,\ldots,h_r]\over (h_1^{n_1+1},\ldots,h_r^{n_r+1})} , \]
where the classes $h_i$ are the pull backs of the hyperplane
classes via the standard projection maps
\[\pi_i:\mathbf{S}\longrightarrow\PP^{n_i}.\]
Let
\[\Lambda_i=\phi_\mathbf{T}(\pi_i^{-1}(H_i)) \subset \PP^n ,
\ \ (i=1,\ldots,r) \ .\]

\begin{lem}\label{SEGREbirationalmap}
The map $\phi_\mathbf{T}$ is a birational map,
\[ \dim_\mathbb{C} \ \mathcal{L}  =n+1 \]
where $n=\sum n_i$,  and the indeterminacy locus is
\[\mathbf{B}=\bigcup_{1\leq i<j\leq r}\pi_i^{-1}(H_i)\cap\pi_j^{-1}(H_j).\]
\end{lem}
\begin{proof}
Consider the coordinate ring of the Segre variety $\mathbf{S}$
\[R=\mathbb{C}[x_0^{(1)},\ldots,x_{n_1}^{(1)},x_0^{(2)},\ldots,x_{n_2}^{(2)},\ldots,x_0^{(r)},\ldots,x_{n_r}^{(r)}]\]
and assume that  $H_i = \{x_0^{(i)}=0\}$,  for $1 \leq i \leq r$.
Then we have that
$\mathcal{L}$
is
the multidegree $(1,\ldots,1)$ part of the ideal
\[(x_0^{(1)},\ldots,x_0^{(r)})^{r-1}\subset R,\]
which is generated by $n+1$ linear independent  forms
\[x_0^{(1)}\cdot\ldots\cdot x_0^{(r)},\]
\[x_1^{(1)}x_0^{(2)}\cdot\ldots\cdot x_0^{(r)},\ldots,
x_{n_1}^{(1)}x_0^{(2)}\cdot\ldots\cdot x_0^{(r)},\]
\[\vdots\]
\[x_0^{(1)}\cdot\ldots\cdot x_0^{(r-1)}x_{1}^{(r)},\ldots,x_0^{(1)}\cdot\ldots\cdot x_0^{(r-1)}x_{n_r}^{(r)}.\]
Hence  the base locus of
$\mathcal{L}$ is
the collection of points $Q_1\times\ldots\times Q_r$ such that at
least two of them, say $Q_i\in\PP^{n_i}$ and $Q_j\in\PP^{n_j}$,  lie on the hyperplanes $H_i $ and $ H_j$, respectively. Thus the base locus is
\[\mathbf{B}=\bigcup_{1\leq i<j\leq r}\pi_i^{-1}(H_i)\cap\pi_j^{-1}(H_j).\]
Following \cite{CGG3}, for $Q \in \mathbf{S}\setminus  \mathbf{B}$, we may write the map
$\phi_\mathbf{T}$
as
\begin{eqnarray}\label{SEGREmapincoordinateHOM}
\phi_\mathbf{T}(Q)& = &
[x_0^{(1)}\cdot\ldots\cdot x_0^{(r)},
x_1^{(1)}x_0^{(2)}\cdot\ldots\cdot x_0^{(r)},
\ldots,
x_{n_1}^{(1)}x_0^{(2)}\cdot\ldots\cdot x_0^{(r)}, \ldots,\\
\nonumber
&  &
x_0^{(1)}\cdot\ldots\cdot x_0^{(r-1)}x_{1}^{(r)},
\ldots,
x_0^{(1)}\cdot\ldots\cdot x_0^{(r-1)}x_{n_r}^{(r)}].
\end{eqnarray}
Hence for points
\[Q=Q_1\times\ldots\times Q_r\in\mathbf{S}\setminus\bigcup_{i=1}^r\pi_i^{-1}(H_i),\]
if we let
$Q_i=[1, q_{1}^{(i)}, \ldots , q_{n_i}^{(i)}]$, the map $\phi_\mathbf{T}$
can be nicely written as
\begin{equation}\label{SEGREmapincoordinate1}
\phi_\mathbf{T}(Q)=[1, q_{1}^{(1)}, \ldots , q_{n_i}^{(1)}, q_{1}^{(2)}, \ldots ,
q_{n_{r-1}}^{(r-1)}, q_{1}^{(r)},\ldots, q_{n_r}^{(r)}] ,
\end{equation}
and this is enough to show that the map is birational (see
\cite{CGG3} for more details).
\end{proof}

\begin{lem}\label{SEGREprojectionmap}
There exists a linear space $\Sigma\subset\PP^N$ such that, for a suitable choice of the
basis of $\mathcal{L}$,
$\phi_\mathbf{T}$ is induced by the projection from $\Sigma$, that is
we have a commutative diagram
\[\xymatrix{
\mathbf{S} \ar@{^{(}->}[r] \ar@{-->}[d]_{\phi_\mathbf{T}}& \mathcal{S}\subset\PP^N\ar@{->}[ld]^{ \pi_\Sigma} \\
\PP(\mathcal{L})^{*}\simeq\PP^n & &}
\]
where $\pi_\Sigma$ denotes the projection map from $\Sigma$.
Moreover, $\Sigma\cap\mathcal{S}$ is the Segre embedding of
$\mathbf{B}$.
\end{lem}
\begin{proof}
By rearranging the forms of (\ref{SEGREmapincoordinateHOM})   in lexicographic order,
we write the map
$\phi_\mathbf{T}$ as
\begin{eqnarray}
\phi_\mathbf{T}(Q)& = &
\nonumber
[x_0^{(1)}\cdot\ldots\cdot x_0^{(r)},x_0^{(1)}\cdot\ldots\cdot
x_0^{(r-1)}x_{1}^{(r)},
\ldots,
x_0^{(1)}\cdot\ldots\cdot x_0^{(r-1)}x_{n_r}^{(r)}, \ldots,\\
\nonumber
&  &
x_1^{(1)}x_0^{(2)}\cdot\ldots\cdot x_0^{(r)},
\ldots,
x_{n_1}^{(1)}x_0^{(2)}\cdot\ldots\cdot x_0^{(r)}],
\end{eqnarray}
where $Q\in\mathbf{S}\setminus\mathbf{B}$. As the Segre embedding
$\mathbf{S}\hookrightarrow\mathcal{S}\subset\PP^N$ is defined by
all the multidegree $(1,\ldots,1)$ forms in the coordinate ring
$R$, it is enough to consider
\[\Sigma=\{x_0^{(1)}\cdot\ldots\cdot x_0^{(r)}=
x_1^{(1)} x_0^{(2)}\cdot\ldots\cdot x_0^{(r)}=
\ldots=
x_0^{(1)}\cdot\ldots\cdot x_0^{(r-1)}x_{n_r}^{(r)}=0\}\subset\PP^N.\] The computation of
$\Sigma\cap\mathcal{S}$ is straightforward.
\end{proof}

As we know that $\phi_{\mathbf{T}}$ is a birational map it is
interesting to investigate subvarieties which are contracted by
this map.

\begin{lem}\label{SEGREcurvelemma}
The space
\[\Lambda_i=\phi_\mathbf{T}(\pi_i^{-1}(H_i))\subset \PP^n ,
( i=1,\ldots,r )\]
is a $n_i-1$ dimensional linear space. Moreover,
 $\Lambda_1, \dots , \Lambda_r$ are  generic linear spaces
in $\PP^n$.
\end{lem}
\begin{proof}
As in the proof of Lemma \ref{SEGREbirationalmap}, we assume that
$\pi_i^{-1}(H_i)$ is defined by $\{x_0^{(i)}=0\}$ on the Segre
variety, and  that $\phi_\mathbf{T}$ is as in (\ref{SEGREmapincoordinateHOM}).

Consider the coordinate ring of $\PP^n$
\[\mathbb{C}[y_0, y_1^{(1)}, \ldots, y_{n_1}^{(1)}, \ldots, y_1^{(r)}, \ldots, y_{n_r}^{(r)}].\]
We  have that
\[\Lambda_1\mbox{ is defined by }\{y_0=y_1^{(2)}=\ldots=y_{n_2}^{(2)}=\ldots=y_1^{(r)}=\ldots=y_{n_r}^{(r)}=0\},\]
\[\vdots\]
\[\Lambda_r\mbox{ is defined by }\{y_0=y_1^{(1)}=\ldots=y_{n_1}^{(1)}=\ldots=y_1^{(r-1)}=\ldots=y_{n_{r-1}}^{(r-1)}=0\}.\]
Thus, the spaces $\Lambda_i$ are linear spaces of dimension
$n_i-1$. Notice that each $\Lambda_i$ is the span of a set of
coordinate points and that each of these coordinate points only
belongs to one $\Lambda_i$. Hence the spaces $\Lambda_i$ are
generic linear spaces in $\PP^n$.
\end{proof}

To bring rational curves into the picture, we recall a rather
elementary technical fact (see  \cite[page 10]{Harris}).
\begin{lem} \label{lemmaRNC}Let $A_1,\ldots,A_m\in\PP^1$ and $Q_1,\ldots,Q_m\in
\PP^t$ be generic points. If $m\leq t+2$, then there exists a
morphism $\psi:\PP^1\longrightarrow\PP^t$ whose image is a
rational normal curve of degree $t$ and such that
$\psi(A_i)=Q_i, \ i=1,\ldots,m$.
\end{lem}
\begin{proof}
It is enough to give the proof for $m=t+2$. For smaller $m$ we can
just add some extra points and repeat the same argument. After a
projectivity, we may assume that  $Q_1$
$=[1,0,\ldots,0],$ $Q_2= [0,1,0\ldots,0]$, $ \dots , Q_{t+1}=[0,\ldots,0,1]$ and $Q_{t+2}=[1,\ldots,1]$.
Then,  we choose $t+1$ linear forms
$L_1,\ldots,L_{t+1}$ on $\PP^1$ such that
$L_i(A_i)=0$, \ ($1 \leq i\leq t+1$), and we consider the map
\[\psi:\PP^1\longrightarrow \PP^t\]
defined for $P \notin \{ A_1, \dots , A_{t+1} \}$ by
\[ P\mapsto \left[{a_1 \over L_1(P)}\ ,\dots,{a_i\over
L_i(P)} \ ,\dots,{a_{t+1}\over L_{t+1}(P)}\right], \]
where $a_i=L_i(A_{t+2})$, ($\ 1 \leq i\leq t+1$), and naturally extended. It is easy to check that
$\psi$ is the required map.
\end{proof}

The previous lemma allows us to construct rational curves with
some special properties on Segre varieties.

\begin{prop}\label{SEGREcurveprop}
Notation as above, let $n_1\leq n_2\ldots\leq n_r$ be positive integers, and let
$\mathbb{X}$ be a set of $s$ generic points on the Segre variety
\[\mathbf{S}=\PP^{n_1}\times\ldots\times\PP^{n_r}.\]
\par
{\rm i)}
If  $s \leq n_1+3$ and $1 < n_1 < n_2$, or $s \leq n_2+2$ and
$n_1=1$ or $n_1=n_2$, then there exists a rational curve
$\mathbf{C}$ such that
\[\mathbb{X}\subset\mathbf{C}\subset\mathbf{S}\]
and $h_i[\mathbf{C}]=n_i$ in the Chow ring $CH(\mathbf{S})$.
\par
{\rm ii)} Let $\mathcal{C}$ be the Segre embedding of
$\mathbf{C}$, then $\deg(\mathcal{C})=n=\sum n_i$ and the linear
span of the curve is a $\PP^n$, i.e. $\mathcal{C}$ is a rational
normal curve in its linear span.
\par
{\rm iii)}
\[\mathrm{deg} (\Lambda_i\cap\phi_\mathbf{T}(\mathbf{C}))=\dim\Lambda_i+1.\] \par
{\rm iv)}
$\phi_\mathbf{T}(\mathbf{C})$ is a rational normal curve of
$\PP^n$.

\end{prop}
\begin{proof}

i) Let  $t  $ denote the maximum number of points such that there
exists an rnc in $\PP^{n_1}$ through $t$ generic points, that is
$t = n_1 +3 $ if $n_1>1$ (see Theorem \ref{Veronese}), and $t =
\infty$ if $n_1=1$. Let $\bar{s}  = \min \{t, n_2+2 \}$. We have
\[ \bar{s} = \min \{t, n_2+2 \} = \left\{
\begin{array}{ccc}
n_1+3 & \ {\rm for} \ 1 < n_1 < n_2 \\
n_2+2 & \ \ \ \ \ \ \ \ \ {\rm for} \ n_1 =1 \ {\rm or} \ n_1 =
n_2  \end{array} \right. ,\] and it suffices to prove the
proposition for $s=\bar{s} $.

Let $\mathbb{X}=\{P_1,\ldots,P_{\bar{s} }\}$ where
\[P_i=Q_i^{(1)}\times\ldots\times Q_i^{(r)}, \ \ ( Q_i^{(j)}\in\PP^{n_j})\]
and the points $Q_1^{(j)},\ldots,Q_{\bar{s} }^{(j)}$ are generic
in $\PP^{n_j}$.

Since $\bar{s}  \leq t$, there exists an embedding
\[\psi_1:\PP^1\hookrightarrow\PP^{n_1}\]
whose image is a rational normal curve passing through the points
$Q_1^{(1)},\ldots,Q_{\bar s}^{(1)}$. Observe that if $n_1=1$, then
we can choose  the identity map as $\psi_1$.

Let $A_i=\psi_1^{-1}(Q_i^{(1)}),\ (i=1,\ldots,{\bar s})$. Since
${\bar s} \leq n_j+2$ for any $j \geq 2$, by Lemma \ref{lemmaRNC}
we  construct embeddings
\[\psi_j:\PP^1\hookrightarrow\PP^{n_j},j=2,\ldots,r\]
such that $\psi_j(A_i)=Q_i^{(j)}$. Now we consider the curve
$\mathbf{C}$ given as the image of the product map
\[
\psi_1\times\ldots\times\psi_r:\PP^1\longrightarrow\mathbf{S}.
\]
To compute $h_i[\mathbf{C}]$ in the Chow ring of $\mathbf{S}$, we
recall that $h_i=[\pi_i^{-1}(H_i)]$ where  $H_i$ is a
hyperplane in $\PP^{n_i}$. Then $h_i[\mathbf{C}]$ is the degree
of $\pi_i(\mathbf{C})=\mbox{Im}(\psi_i)\subset\PP^{n_i}$ and this
is $n_i$.

If we let $\mathcal{C}\subset\PP^N$ be the Segre embedding of
$\mathbf{C}$, then $\deg(\mathcal{C})$ can be determined by
computing in $CH(\mathbf{S})$, i.e. we have to evaluate
\[(h_1+\ldots+h_r)[\mathbf{C}]=\sum h_i[\mathbf{C}]=\sum n_i=n. \]

ii)
To show that the linear span $\langle\mathcal{C}\rangle$ of $\mathcal{C}$ has dimension $n$,
recall that the curve $\mathcal{C}$ is the Segre embedding of $\mathbf{C}$,
and that, for each $i$,  $\psi_i(\PP^1)$  is a rational normal curve.
Now let $\mathtt{R}=\CC[x,y]$ be the coordinate ring of $\PP^1$, then $\psi_i(\PP^1)$
is the image of the linear system $\mathtt{R}_{n_i}$ of the forms of degree $n_i$,
and the  multilinear multiplication map
\[\mathtt{R}_{n_1}\times\ldots\times \mathtt{R}_{n_r}\longrightarrow \mathtt{R}_{n_1+\ldots
+n_r} =\mathtt{R}_n\]
is obviously surjective.
\par
iii)
As $\Lambda_i=\phi_\mathbf{T}(\pi_i^{-1}(H_i))$ we can work in the
Chow ring of $\mathbf{S}$ where
\[[\Lambda_i\cap\phi_\mathbf{T}(\mathbf{C})]=h_i[\mathbf{C}]=n_i=\dim\Lambda_i+1.\]
\par

iv)
To show that $\phi_\mathbf{T}(\mathbf{C})$ is an rnc it is enough
to prove that
\[ \langle\mathcal{C}\rangle  \cap\Sigma=\emptyset , \]
or, equivalently, to prove that
$\phi_\mathbf{T}(\mathbf{C})$ does not lie on a
hyperplane, i.e. it is not degenerate.

Assume that  $\phi_\mathbf{T}(\mathbf{C})$ lies on a hyperplane
$H$. Since $ \mathbb{ X }$ is a set of generic points, then $H
\neq   \langle \Lambda_1, \dots, \Lambda_r \rangle $, hence there
exists at least one $\Lambda_i$ such that the points of
\[\Lambda_i\cap\phi_\mathbf{T}(\mathbf{C})\]
are degenerate in $\Lambda_i$,
and this is the case if and only if
the points of
\[H_i\cap\mbox{Im}(\psi_i)\]
are degenerate in $H_i$, in fact, by \eqref{SEGREmapincoordinate1}, the spaces $\Lambda_i$ are
naturally identified with the hyperplanes $H_i$.
Since
$\mbox{Im}(\psi_i)$ is an rnc, we get a contradiction.
\end{proof}

\begin{rem}\label{SEGREP1remark}
Note that for  $n_1=1$ the proof of the previous proposition does
not require Veronese's  Theorem.
\end{rem}

\begin{rem}\label{alternative proof}
We give a sketch of an alternative   proof of  Proposition \ref
{SEGREcurveprop} iv).

With the notation as above, let $Q$ be a point in $\PP^{n_j}$ and
set
\[ \psi _j (Q) = (f^{(j)}_0 |_Q, \dots , f ^{(j)} _{n_j}  |_Q  )\in \PP^{n_j} ,  \]
where $f_0^{(j)}, \dots , f ^{(j)} _{n_j}   $ is a base of  $\mathtt{R}_{n_j}$.
Then by (\ref{SEGREmapincoordinateHOM}) we get that the curve $\phi_\mathbf{T}(\mathbf{C})$ is the image of the linear system:
\[ (
f _0^{(1)} \cdot\ldots\cdot  f _0^{(r)}  , \ \
f_1^{(1)}  f _0^{(2)}  \cdot\ldots\cdot f_0^{(r)},
\ldots,
f _{n_1}^{(1)}  f _0^{(2)}\cdot\ldots\cdot f _0^{(r)}, \ldots, \ \ \]
\[
f _0^{(1)}  \cdot\ldots\cdot f _0^{(r-1)}   f _{1}^{(r)} ,
\ldots,
f _0^{(1)} \cdot\ldots\cdot f _0^{(r-1)}  f _{n_r}^{(r)}  ) \ \subset \
\mathtt{R}_n . \]

If we prove that  these $n+1$ forms are linearly independent, then we are done.
\par

The key point to do this is the following:
there does not exist a point $A \in \PP^1$ such that
$f _0^{(i)}|_A =f _0^{(j)}|_A = 0$, for $i \neq j$, and this is a (non-obvious) consequence of
the genericity of  $\mathbb{X}$.

\end{rem}

Now our goal is to use the curves of the proposition above to address
the problem of feasibility.

\begin{thm}\label{SEGREteo}
Let $n_1\leq\ldots\leq n_r$ be positive  integers, and set
$n=n_1+\ldots+n_r$. Consider in $\PP^n$ a set of $s$ generic points
$P_1,\ldots,P_s$ and $r$ generic linear spaces
$\Lambda_1,\ldots,\Lambda_r$ of dimension $n_1-1,\ldots,n_r-1$,
respectively.

If $s \leq n_1+3$ and $1 < n_1 < n_2$, or $s \leq n_2+2$ and
$n_1=1$ or $n_1=n_2$, then the configuration of linear spaces
\[\Lambda_1\cup\ldots\cup\Lambda_r\cup\{P_1,\ldots,P_s\} \subset \PP^n\]
is feasible.
\end{thm}

\begin{proof}
Consider the Segre variety $\PP^{n_1}\times \dots \times
\PP^{n_r}$ and choose hyperplanes $H_i \subset \PP^{n_i}$ such
that the birational map
\[\phi_\mathbf{T} : \PP^{n_1}\times \dots \times \PP^{n_r}\dashrightarrow \PP^n , \]
where $\mathbf{T}=H_1\times\ldots\times H_r$,
contracts
\[H_1\times\PP^{n_2}\times \dots \times \PP^{n_r},\]
\[ \PP^{n_1} \times H_2\times\PP^{n_3}\times \dots \times \PP^{n_r}, \]
\[ \vdots  \]
\[ \PP^{n_1} \times \dots \times\PP^{n_{r-1}}\times H_r\]
onto the linear spaces
$\Lambda_1,\ldots,\Lambda_r$ ,  respectively.

Chosen generic points  $P_1,\ldots ,P_s\in\PP^n$, via the map
$\phi_\mathbf{T}$ these points correspond to $s$ generic points on
the Segre variety. Then, the image of the curve of Proposition
\ref{SEGREcurveprop} is an rnc in $\PP^n$ passing through the
points $P_i$, and $n_i-1$-secant to $\Lambda_i$. Hence, the
generic configuration of linear spaces
\[\Lambda_1\cup\ldots\cup\Lambda_r\cup\{P_1,\ldots,P_s\} \subset \PP^n\]
is feasible.
\end{proof}

\begin{ex}\label{P2P3ex}
We illustrate our strategy explicitly by studying the weight
vector $L=(5,1,1,0)$ in $\PP^5$.
Let $P_1,\ldots ,P_5\in\PP^5$ be generic points, and $\Lambda_1 , \Lambda_2 \subset \PP^5$
be generic linear spaces, of dimension $1$ and $2$,  respectively.
Consider the Segre variety
$\PP^2\times\PP^3$ and let  $\mathbf{T} \subset \PP^2\times\PP^3$,
$\mathbf{T}=H_1\times
H_2\simeq\PP^1\times\PP^2$ be such that the  birational map
\[\phi_\mathbf{T}:\PP^2\times\PP^3\dashrightarrow \PP^5\]
contracts
$H_1\times\PP^3$, and $\PP^2\times H_2$, onto the line
$\Lambda_1$, and the plane $\Lambda_2$,  respectively.
Via the map $\phi_\mathbf{T}$, the points $P_i$ corresponds to $5$ generic points
in $\PP^2\times\PP^3$. So
the image, via the map
$\phi_\mathbf{T}$, of the curve in Proposition \ref{SEGREcurveprop}
is an rnc in $\PP^5$ passing through the points $P_i$, 2-secant to
$\Lambda_1$ and 3-secant to $\Lambda_2$. Hence $L$
is feasible.
\end{ex}

\begin{rem}\label{SEGREVERONESErem}
The previous example is related to the study of the variety of
$5$-secant $\PP^4$'s to the Segre-Veronese variety
$\PP^2\times\PP^3$ embedded using forms of bidegree $(1,2)$. As
noticed in \cite[Subsection 6.2]{CaCa}, the  existence of the rnc
yields  that this Segre-Veronese variety is $5$-defective.
\end{rem}

\begin{rem}\label{VERONESEremark}
We notice that Theorem \ref{SEGREteo} and Remark
\ref{SEGREP1remark}  allow us to recover Veronese's Theorem in
$\PP^n$. It is enough to consider $r=n$, $n_1=\ldots=n_r=1$, and
$s=3$ to show that the weight vector $(n+3,0,\ldots,0)$ in $\PP^n$
is feasible.
\end{rem}

\begin{rem}
Another proof of the feasibility of $(n+2,0,\ldots,0,1)$ in
$\PP^n$ can be deduced by Theorem \ref{SEGREteo} with $r=2$,
$n_1=1$, $n_2=n-1$, and $s=n_2+2=n+1$ generic points  (see also
Theorem \ref{NOSTROcod2}).
\end{rem}

\section{Methods for non-feasibility}\label{SECnonfeasible}

\subsection{Parameter space}\label{paramterSUBSEC}
We can use a natural parameter space for rncs to get a powerful
non-feasibility condition. Given a weight vector $L$ in $\PP^n$ we
consider a parameter space for the configurations of linear spaces
$\Lambda_L$, namely
\[\mathcal{D}=\left(G(\PP^0,\PP^n)\right)^{l_0}\times\ldots\times\left(G(\PP^i,\PP^n)\right)^{l_i}\times\ldots\times\left(G(\PP^{n-2},\PP^n)\right)^{l_{n-2}},\]
where  $G(\PP^i,\PP^n)$ is the  Grassmannian of the $i$-dimensional linear subspaces of $\PP^n$.
We also recall that there exists an irreducible component of the
Hilbert scheme, namely
\[\mathcal{H}\subset\mathrm{Hilb }_{nt+1}(\PP^n),\]
which is a parameter spaces for rncs in $\PP^n$. Finally, we consider the
incidence correspondence
\[\mathcal{H}\times\mathcal{D}\supset\Sigma=\left\lbrace (\mathcal{C},\Lambda_L) : \mathcal{C}\mbox{ maximally intersects }\Lambda_L
\right\rbrace\] and the natural projection map
$\psi:\Sigma\longrightarrow\mathcal{D}$. Clearly we have that
$L$ is feasible if and only if $\psi$ is a dominant map.

\begin{prop}\label{numericalNOTEXIST}
Let $L=(l_0,\ldots,l_{n-2})$ be a weight vector in $\PP^n$. If
\[\sum_{i=0}^{n-2} (i+1)(n-1-i)l_i> (n+3)(n-1),\]
then $L$ is non-feasible.
\end{prop}

\begin{proof} With the notation as above, we compute
$\dim\Sigma$. In order to do this, we use the projection
$\Sigma\longrightarrow\mathcal{H}$ whose fibers  have
dimension
\[l_0+2l_1+\ldots +(i+1)l_i+\ldots +(n-1)l_{n-2}=\sum_{i=0}^{n-2}(i+1)l_i.\]
As $\dim\mathcal{H}=(n+3)(n-1)$, we get
\[\dim\Sigma=(n+3)(n-1)+\sum_{i=0}^{n-2}(i+1)l_i.\]
Now we compute $\dim \Sigma-\dim \mathcal{D}$. We have:

\[\dim \Sigma-\dim \mathcal{D}=(n+3)(n-1)+\sum_{i=0}^{n-2}(i+1)l_i-\sum_{i=0}^{n-2}(i+1)(n-i)l_i\]
\[=(n+3)(n-1)-\sum_{i=0}^{n-2}(i+1)(n-1-i)l_i <0\]
Thus $\psi$ is not a dominant map, and  $L$ is non-feasible.
\end{proof}

A priori, there are infinitely many weight vectors that one should
study for feasibility in each $\PP^n$,  but from the  previous
proposition we immediately get the following

\begin{cor}\label{finitelyCOR}
There are only finitely many feasible weight vectors in $\PP^n$.
\end{cor}
\hfill \qed

\begin{rem} The numerical conditions of the Proposition above is
sufficient, but it is not necessary for feasibility. For instance  consider the weight
vector
$L=(4,2)$ in $ \PP^3$.
By Theorem \ref{NOSTROcod2} we know that $L$ is non-feasible, see also
\cite{CaCa}. However, one has
\[\sum_{i=0}^{n-2} (i+1)(n-1-i)l_i=12=(n+3)(n-1)\]
and non-feasibility can not be obtained by Proposition
\ref{numericalNOTEXIST}.
\end{rem}

\subsection{Bezout argument}

A simple elaboration of a Bezout type argument provides a useful
tool to prove non-feasibility in many interesting situations, for
example when Proposition \ref{numericalNOTEXIST} cannot be
applied.

The starting point is a consequence of the Riemann-Roch Theorem
for rational curves. Namely, given  points $P_i $ with
multiplicity $m_i$, and a degree $n$ very ample divisor $H$ on
$\PP^1$ one has
\[h^0(\mathcal{O}_{\PP^1}(dH-\sum m_i P_i))=nd+1-\sum m_i .\]
Hence, a degree $d$ hypersurface meeting a degree $n$ rnc in more
than $nd$ points, counted with multiplicity, must contain the
curve.

\begin{prop}\label{bezoutPROP} Let   $\Lambda'\subset\PP^n$ be a generic configuration  of linear
spaces, and let $\Gamma \subset\PP^n$ be a generic $k$-dimensional linear space.
Construct the configuration $\Lambda=\Lambda'\cup
\Gamma$ and let $L=(l_0,\ldots,l_{n-2})$ be its weight vector.
If,
for an integer d, the following Hilbert function equality holds
\[H(\Lambda,d)=H(\Lambda',d)+{d+k\choose k},\]
and moreover
\[\sum_{i=0}^{n-2}l_i(i+1)-1>dn,\]
then $L$ is non-feasible.
\end{prop}
\begin{proof}
Assume  that $\Lambda$ is feasible and let
$\mathcal{C}$ be a required rnc. Consider the degree $k+1$
scheme $\mathcal{C}\cap\Gamma$ which, by the genericity of
$\Lambda$, is reduced and can be written as
\[\mathcal{C}\cap\Gamma=\{P_1,\ldots,P_{k+1}\}.\]
Since the $k+1$ points $P_i$ are in  linear general position,
if we choose
$t={d+k\choose k}-(k+1)$ generic points $Q_i\in\Gamma$, then any  hypersurface of degree $d$
passing through the ${d+k\choose k}$ points $P_1, \dots,P_{k+1},Q_1,\dots,Q_t$
has $\Gamma$ as fixed component.
Hence
\[\dim (I_{\Lambda'\cup P_1\cup\ldots\cup P_{k}\cup Q_1\cup\ldots\cup Q_t})_d=
\dim (I_{\Lambda'\cup\Gamma})_d=
\dim (I_{\Lambda})_d=\dim ( I_{\Lambda'})_d-{d+k\choose k}.\]
It follows that  the $k+1$ points $P_i$ impose independent conditions to the forms of
$(I_{\Lambda'})_d$.
Now choose $k$ of them, say $P_1, \dots, P_k$ ,  let $F\in (I_{\Lambda'\cup P_1\cup\ldots\cup P_{k}})_d$ and
notice that
\[\deg(\{F=0\}\cap\mathcal{C})=\sum_{i=0}^{n-2}l_i(i+1)-(k+1)+k>dn ,\]
hence $\{F=0\}\supset\mathcal{C}$.
Since $P_{k+1}\in\mathcal{C}$, we get
\[ (I_{\Lambda'\cup P_1\cup\ldots\cup P_{k}})_d= (I_{\Lambda'\cup P_1\cup\ldots\cup P_{k}\cup P_{k+1}})_d\]
and this is a contradiction.
\end{proof}

\begin{ex} Proposition \ref{bezoutPROP} provides the non-feasibility  in  $\PP^n$ of the weight vector
$L=(4,0,\ldots,0,2)$ using $k=0$, that is $\Gamma $ is a point, and $d=2$ (see also
Theorem \ref{NOSTROcod2}).
\end{ex}

\subsection{Projection argument}
We present a simple argument which allows us to somehow improve
Theorem \ref{SEGREteo} into a necessary and sufficient statement.

\begin{lem} \label{PROJlemma} Let $\Lambda\subset\PP^n$ be a configuration of
linear spaces. Let $\Lambda_i$ be an $i$-dimensional component of
$\Lambda$, consider the projection $\pi$ from $\Lambda_i$
\[
\pi:\PP^n\longrightarrow\PP^{n-i-1}
\]
and let $\Lambda'=\pi(\Lambda)$ be the projected configuration. If
$\Lambda'$ is non-feasible, then $\Lambda$ is non-feasible also.
\end{lem}
\begin{proof} Assume $\Lambda$ is feasible and let $\mathcal{C}$ be an
rnc with the required properties. Then, $\pi(\mathcal{C})$ is a
rational normal curve of $\PP^{n-i}$ maximally intersecting
$\Lambda'$, a contradiction.
\end{proof}

\begin{rem} If the notation is as in the lemma above, notice that if
$\Lambda$ is generic then $\Lambda'$ is also. In particular, the
conclusion can be read in term of weight vectors.
\end{rem}

\begin{thm}\label{SEGREteoiff}
Let $n_1\leq\ldots\leq n_r$ be positive  integers, and set
$n=n_1+\ldots+n_r$. Consider in $\PP^n$ a set of $s$ generic points
$P_1,\ldots,P_s$ and $r$ generic linear spaces
$\Lambda_1,\ldots,\Lambda_r$ of dimension $n_1-1,\ldots,n_r-1$,
respectively. \par
The
configuration of linear spaces
\[\Lambda=\Lambda_1\cup\ldots\cup\Lambda_r\cup\{P_1,\ldots,P_s\}\]
is feasible if and only if
 $s \leq n_1+3$ and $1 < n_1 < n_2$, or $s \leq n_2+2$ and
$n_1=1$ or $n_1=n_2$.
\end{thm}

\begin{proof}
By Theorem \ref{SEGREteo}, we only have to show that the numerical
conditions on $s$ are necessary, more precisely, we have to prove
that if $s \geq n_1+4$, or $s \geq n_2+3$ and $n_1=1$ or
$n_1=n_2$, then $\Lambda$ is non-feasible.

Case:  $s \geq n_1+4$.
Consider the projection
from the linear span
\[\langle\Lambda_2\cup\ldots\cup\Lambda_r\rangle\]
mapping $\PP^n$ onto $\Pi_1 \simeq \PP^{n_1}$ and the $s$ points onto $s$ generic points
in $\Pi_1$ . Since  in $\PP^{n_1}$
there is not an rnc through $s \geq n_1+4$ generic points,  we conclude by Lemma \ref{PROJlemma}.

Case: $s \geq n_2+3$ and $n_1=1$.
In this case,  $\Lambda_1$ is a point, say $P$. Consider the projection
from the linear space
\[\langle \Lambda_3\cup \ldots\cup\Lambda_r\rangle,\]
whose dimension is $(n_3+ \ldots +n_r )= n-n_2-1$,
mapping $\PP^n$ onto $\Pi_2 \simeq \PP^{n_2+1}$,
and the  $s+1$ points $P,P_1, \ldots, P_s$ and  $\Lambda_2$  onto a generic configuration
\[ P' \cup P'_1 \cup \ldots  \cup P'_s \cup \Lambda'_2 \subset \Pi_2 . \]
Since  in $\Pi_2$ the configuration above is non-feasible
(see Proposition \ref{numericalNOTEXIST}), then  we conclude by Lemma \ref{PROJlemma}.

Case $s \geq n_2+3$, and  $n_1=n_2 $.  Set  $m=n_1=n_2 $.

By projecting $\PP^n$ from the linear spaces $
\langle\Lambda_2\cup\ldots\cup\Lambda_r\rangle $ and $
\langle\Lambda_1\cup\Lambda_3\cup\ldots\cup\Lambda_r\rangle $ onto
$ \Pi_3 \simeq \PP^{m}$, one sees that the images of the points
$P_i$ give two sets of $s \geq m+3$ generic points in $\PP^m$,
which are not projectively equivalent, but they would be if
$\Lambda$ were feasible.

\end{proof}

\begin{ex} As a straightforward application of the previous
Theorem, we have that the weight vector $(p,1,1,0)$ is feasible if
and only if $p\leq 5$.
\end{ex}

\section{Applications}\label{SECapp}

\subsection{Homogenous configurations}\label{HOMOGENEOUSsec}
We can focus our attention on homogeneous configurations of linear
spaces, i.e.  configurations whose components all have the same
dimension; notice that this is the setting of Veronese's Theorem
\ref{Veronese} where only points are involved. In this simplified
situation, Propositions \ref{numericalEXIST} and
\ref{numericalNOTEXIST} are enough to determine whether $\Lambda$
is feasible or non-feasible in almost all cases. More precisely,
the following proposition shows that for $n\gg 0$ there is at most
one case not covered.

\begin{prop} \label{homoprop}
Let $n \geq 3$, and let $\Lambda\subset\PP^n$ be a  homogeneous
configuration of
$l$ generic
$i$-dimensional linear spaces.  If  $n> i^2 +5i +1$, then
 $\Lambda$ is feasible for
$l \leq   { {n+3} \over {i+1}}   $, and $\Lambda$ is non-feasible
for   $l >  \lceil{ {n+3} \over {i+1}} \rceil $.
\end{prop}

\begin{proof}
By Propositions \ref{numericalEXIST}
and \ref{numericalNOTEXIST} we have only to consider the values of $l$ such that
\[  {n+3 \over i+1}  <l \leq {(n+3)(n-1)\over (i+1)(n-1-i)} ,\]
as Proposition \ref{numericalEXIST} gives feasibility for $l \leq {n+3 \over i+1}$ and
Proposition \ref{numericalNOTEXIST} shows non-feasibility for
$l> {(n+3)(n-1)\over (i+1)(n-1-i)}$.
\par
 It is easy to show that,  for $n> i^2 +5i +1$, we have
 \begin{eqnarray}\label{DISEQ}
   {(n+3)(n-1)\over (i+1)(n-1-i)} -  {n+3 \over i+1}   <1  .
    \end{eqnarray}

  Hence if   $i+1$  divides $n+3$,  then
 $l> {(n+3)(n-1)\over (i+1)(n-1-i)}$ is equivalent to
 $l >  \lceil{ {n+3} \over {i+1}} \rceil = { {n+3} \over {i+1}}$, and the conclusion follows.
 \par
 If $i+1$ does not divide $n+3$, let $q, r   \in  \Bbb N$ be such that
 \[   n+3 = q (i+1) + r  \ , \ \ \ \   0 <  r  < i +1 \  .
 \]
  By a direct computation we get
  \[   {(n+3)(n-1)\over (i+1)(n-1-i)} -  q   > 1  ,\]
  hence by (\ref{DISEQ})  we have

 \[
1 + q <  {(n+3)(n-1)\over (i+1)(n-1-i)}  <1 +  {n+3 \over i+1}   = 1+q+ {r \over {i+1}}
 \]
  then
 $l> {(n+3)(n-1)\over (i+1)(n-1-i)}$ is equivalent to
 $l > q+1 =  \lceil{ {n+3} \over {i+1}} \rceil $, and we are done.
  \end{proof}

If the components of $\Lambda$ are all lines, we have an
interesting result, which covers all cases except for
$(n,l)=(5,5)$ and  $(n,l)=(7,6)$, that is it is not know whether
configurations of five generic lines in $\PP^5$, or of six generic
lines in $\PP^7$ are feasible or not.

\begin{prop}\label{lineprop}
Let  $n \geq 3$,  let $\Lambda\subset\PP^n$ be the union of $l$ generic
lines,
then
\par
{\rm i)} $\Lambda$ is feasible in the following cases:
\begin{itemize}
\item
$n =3$ and $l \leq 6$;
\item   $n >3$  and $l \leq \lfloor {n\over 2} \rfloor+2$;
\end{itemize}
\par
{\rm ii)} $\Lambda$ is non-feasible in the following cases:
 \begin{itemize}
\item $n =3$ and $l \geq 7$;
\item $n=5$ and $l\geq 6$;
\item  $n=7$ and $l \geq 7$;
\item $n >3$, $n \neq5$,  $n \neq7$ and $l \geq \lfloor {n\over 2}
\rfloor+3$.

\end{itemize}
\end{prop}

\begin{proof}
{\rm i)} For  $n =3$ see
\cite{Wakeford}, or  \cite[Proposition 3.1]{CaCa}.
\par
If  $n>3$ is odd, by Proposition
\ref{numericalEXIST} we have that $\Lambda$ is feasible if $2l\leq n+3$,
that is for $l\leq \lfloor {n\over 2} \rfloor+2$. \par
If $n>3$ is even, by applying Theorem \ref{SEGREteoiff}
with $s=4$, $r= {n\over 2} $, $n_1= \ldots =n_t =2$, we
have that a generic configuration of ${n\over 2}$ lines and $4$ points
\[\Lambda_1\cup\ldots\cup\Lambda_{n\over 2}\cup\{P_1,\ldots,P_4\} \subset
\PP^n
\]
is feasible. Hence by choosing $P_1,P_2$ on a  generic line and
$P_3,P_4$ on another  generic line, we get that $\Lambda$ is
feasible for  $l = {n\over 2}+2$, and so for  $l \leq {n\over
2}+2$ also.
\par
{\rm ii)} For $n >7$, by Proposition \ref{homoprop} we immediately
get the conclusion. \par Let $n \leq 7$. By Proposition
\ref{numericalNOTEXIST} we have that the configuration is
non-feasible if $l> {{(n+3)(n-1)} \over {2(n-2)}}$, so the
conclusion easily follows for $n\neq 4 $, or for $n= 4 $ and $l
\geq 6$.
\par
Finally for the case $(n,l)=(4,5)$,  use the Bezout argument of
Proposition \ref{bezoutPROP} with $d=2$ and $k=1$. Notice that in
this case it is crucial to have the knowledge of the Hilbert
function of a generic set of lines, see for example
\cite{HartHir}.
\end{proof}

\subsection{One space in each dimension} Here we consider configurations of linear spaces having one component in each possible dimension.

\begin{prop}\label{onespaceineachPROP}
Consider the weight vector $L=(1,\ldots,1)$ in $\PP^n$.  Then $L$
is feasible for $n\leq 5$ and non-feasible for $n\geq 8$.
\end{prop}
\begin{proof}
We apply Propositions \ref{numericalEXIST} and \ref{numericalNOTEXIST}
to get that $L$ is feasible for
\[n+3\geq\sum_{i=0}^{n-2}(i+1)={n \choose 2}\]
and $L$ is non-feasible for
\[(n+3)(n-1)<\sum_{i=0}^{n-2}(i+1)(n-i-1)={n+1 \choose 3}.\]
Solving these inequality we are left with the cases $5\leq n\leq 7$.
\par
By applying Theorem \ref{SEGREteo} with $r=2$, $n_1=2$ and $n_2=3$
we get that the following generic configuration is feasible
\[  \Lambda_1 \cup \Lambda_2 \cup \{P_1, \ldots,P_5\} \subset \PP^5,
\]
where $ \Lambda_1$ is  a line,  $\Lambda_2$ is a plane, and the $P_i$ are
points.
Hence by choosing $P_1,P_2,P_3,P_4$ on a  generic linear space of dimension $3$, we
get that $L=(1,1,1,1)$ is feasible in $\PP^5$.
\end{proof}

We notice again that, even in this simplified situation, there are
cases which we can not solve.

\subsection{The weight vectors $\mathbf{(0,3,0,\ldots,0,2,0)} $ are non-feasible}
In this subsection we study configurations of three lines and two spaces
of codimension 3.  In the particular case of $\PP^4$, we
have already proved  the non-feasibility of five lines (see Proposition
\ref{lineprop}). The following proposition generalizes that result.

\begin{prop}
The weight vector $L=(0,3,0,\ldots,0,2,0)$ in $\PP^n$, $n \geq 5$,  is
non-feasible.
\end{prop}

\begin{proof}
Let
\[\Lambda_L =\Gamma_1 \cup \Gamma_2 \cup \Gamma_3 \cup \Lambda_1 \cup \Lambda_2  \subset \PP^n\]
be a generic configuration of
linear spaces of weight $L$, where
$\Gamma_i \simeq \PP^1$  , $i=1,2,3$, and $\Lambda_j \simeq \PP^{n-3}$ ,
$j=1,2$, and let
\[\Lambda ' = \Lambda_L  \setminus  \Gamma_1  =\Gamma_2 \cup \Gamma_3 \cup
\Lambda_1 \cup \Lambda_2.\] By Proposition \ref{bezoutPROP} we are done if we prove
that there exists an integer $d$ such that
\[\sum_{i=0}^{n-2}l_i(i+1)-1>dn,\]
and
\[H(\Lambda_L,d)=H(\Lambda',d)+{d+1\choose 1}.\]
Let $d=2$. Since:
\[\sum_{i=0}^{n-2}l_i(i+1)-1-dn =  6 +  2(n-2)-1-2n =1 >0 \]
it suffices to show that a generic  line imposes $3$ independent
conditions to the hyperquadrics containing $\Lambda '$.

It is easy to prove that
\[H(\Lambda_1 \cup \Lambda_2,2)=
H(\Lambda_1, 2)+ H(\Lambda_2, 2) - H(\Lambda_1 \cap \Lambda_2,2), \]
so, since for $n=4,\ 5$ we have  $\Lambda_1 \cap \Lambda_2 = \emptyset$, and for $n>5$ we have $\Lambda_1 \cap \Lambda_2 \ \simeq \PP^{n-6}$, we get
\[H(\Lambda_1 \cup \Lambda_2,2)= 2 {{n-3+2} \ \choose 2} - {{n-6+2} \ \choose 2} ={{n^2+3n-16}\over 2}. \]
Now  we expect that three generic lines impose $9$ conditions to
the hyperquadrics  containing $\Lambda_1 \cap \Lambda_2$, in other
words, we  expect that the  Hilbert function of $\Lambda_L$ in
degree $2$ is
\[ {\rm expected} \ H(\Lambda_L,2) = 9+ {{n^2+3n-16}\over 2 }= {{n+2} \choose 2}. \]
So, if we prove that there are not forms of degree $2$ in the ideal $I _{\Lambda_L}$, then
it  follows that  there are exactly three forms of degree two in  ${(I _{\Lambda'})}_2$, and
we are done.

We work by induction on $n$.  \par

For $n=5$, first specialize  the three  lines $\Gamma_i$ onto a
hyperplane $H =\{f=0\}$,

Consider the exact sequence,
\[  0  \longrightarrow (I _{\Lambda_L \setminus H})_1 \longrightarrow  (I _{\Lambda_L})_2
\longrightarrow  (I _{\Lambda_L \cap H})_2 ,
 \]
 where the first map is given by the multiplication by $f$, and
 where $\Lambda_L \setminus H$ is the residual of $\Lambda_L$
 with respect to $H$, and  $\Lambda_L \cap H$ is  the trace of $\Lambda_L $ on $H$.
 \par
Since  $\Lambda_L \setminus H = \Lambda_1 \cup \Lambda_2$,  and it
does not lie on a hyperplane, then $\dim (I _{\Lambda_L \setminus
H})_1=0$. Moreover, since $\Lambda_L \cap H$ consists of five
lines in $\PP^4$, we get $\dim (I _{\Lambda_L \cap H})_2 =0$ (see
\cite{HartHir}). The conclusion follows from the  semicontinuity
of the Hilbert function.
\par
Now assume $n>5$ and
let $H=\{f=0\}$ be a hyperplane containing the three lines $\Gamma_i$.
Consider the following exact sequence analogous to the one above,
\[  0  \longrightarrow (I _{\Lambda_L \setminus H})_1 \longrightarrow  (I _{\Lambda_L})_2
\longrightarrow  (I _{\Lambda_L \cap H})_2 .
 \]
Since  $\Lambda_L \setminus H = \Lambda_1 \cup \Lambda_2$ does not lie on a
hyperplane, then $\dim (I _{\Lambda_L
\setminus H})_1=0$. Moreover $\dim (I _{\Lambda_L \cap H})_2 =0$
follows by the inductive hypothesis, so we get   $\dim (I
_{\Lambda_L })_2 =0$ and we are done.
\end{proof}

\subsection{A new result about defectiveness  of higher secant
varieties of  Segre-Veronese varieties.} As noted in Remark
\ref{SEGREVERONESErem} the existence of a suitable rnc  yields
that  the Segre-Veronese variety $\PP^2\times\PP^3$ embedded using
forms of bidegree $(1,2)$ is $5$-defective. By Theorem
\ref{SEGREteo}  we get  the following new result involving
Segre-Veronese varieties, for more on these varieties see
\cite{CGG3,CGG4}.

\begin{cor}\label{SEGREVERONESEcor}
Let $V \subset \PP^N$ be the Segre-Veronese variety $\PP^1 \times \PP^m  \times \PP^m$
embedded with multi-degree
$(2,1,1)$. Let   $V^s$ be the $s^{th}$  higher secant variety of $V$, that is  the closure
 of the union of all linear spaces spanned by s independent points of $V$. If
\[m+2 \leq s  \leq 2m+1 ,
 \]
then $V$ is $s$-defective, i.e. $V^s$ has not the expected dimension.

  \end{cor}
  \begin{proof}   For $\lceil { { 3(m+1)} \over 2}  \rceil  \leq s \leq2 m+1$  we already know
 that  $V$ is $s$-defective (see \cite[Section 3]{CGG3}).
 In fact the expected dimension of $V^s$ is $N$, but  there exist  form $f_1$ and $f_2$ of
multidegree
$(1,1,0)$ and  $(1,0,1)$, respectively,  passing through $s$ generic (simple) points,
 hence $f_1f_2$ is a form  of multidegree $(2,1,1)$  through $s$ generic 2-fat points, which was not supposed to exist.
\par
The case  $ m+2 \leq s \leq \lfloor { { 3(m+1)} \over 2}  \rfloor  $ was not previously known.
For such $s$  the expected dimension of $V^s$ is
\[ \exp \dim V^s  =N - s (2m+2)
.\]
Now
\[  \dim V^s =
\dim (I_{W\cup Z})_4
,\]
where  $W$ is a generic subscheme of $ \PP^{2m+1}  $ consisting of one double point $Q$ and
 two triple linear spaces $\Lambda_1, \Lambda_2$ of dimension $m-1$ , while $Z$ consists of $s$
 generic double points $P_1,
\dots,P_s$   (see \cite[Theorem 1.5]{CGG3}). Since  $\dim (I_{W})_4  =N =  3(m+1)^2 $, we have that
\[\dim (I_{W\cup Z})_4  =N - h , \]
where $h$ is the number of independent conditions imposed by the $s$ double points $P_i$ to
 the forms  of $(I_{W})_4$. Note that the expected value for $h$ is $s(2m+2)$.

By Theorem \ref{SEGREcurveprop} we know that there exists an rnc $\mathcal C$ maximally intersecting
the configuration $\Lambda_1 \cup  \Lambda_2 \cup Q \cup P_1 \cup \dots \cup P_s$. Let
$ F \in (I_{W\cup Z'})_4$, where $Z'$ consists of the $s-1$ double points  $P_1, \dots,P_{s-1} $ and
the simple point $P_s$. We have
\[ {\rm degree} \{F=0\} \cap \mathcal C \geq 6m+2s +1 \geq  6m+2(m+2)+1 =8m+5, \]
so by Bezout Theorem  $\mathcal C \subset \{ F=0\} $. Hence the
double point $P_s$ imposes less then $2m+2$ conditions to the
forms of $(I_{W\cup Z})_4$, so
\[  \dim V^s =   \dim (I_{W\cup Z})_4 =N-h> N -s(2m+2) =\exp \dim V^s
,\]
and the conclusion follows.
  \end{proof}

\bibliographystyle{alpha}
\bibliography{carlininew}

\end{document}